\documentclass[11pt,twoside]{amsart}

\usepackage{amsmath, amsthm, amscd, amsfonts, amssymb, graphicx, color}
\usepackage[bookmarksnumbered, plainpages]{hyperref}
\usepackage[capitalise]{cleveref}
\usepackage{comment}
\usepackage{calrsfs}
\usepackage{tikz-cd}
\usepackage{capt-of}

\usepackage[skins,theorems]{tcolorbox}
    \tcbset{highlight math style={size=fbox,
    colframe=red,colback=white,arc=2pt,boxrule=0.5pt}}

\usepackage{tikz}
\usetikzlibrary{arrows,decorations.markings,decorations.pathreplacing,calc,matrix,intersections}
\tikzset{->-/.style={decoration={markings,mark=at position #1 with {\arrow{>}}},postaction={decorate}}}

\newtheorem{thm}{Theorem}[section]
\newtheorem{cor}[thm]{Corollary}
\newtheorem{lem}[thm]{Lemma}

\newtheorem{defn}[thm]{Definition}

\newtheorem*{thmA}{Theorem A}
\newtheorem*{thmB}{Theorem B}
\newtheorem*{conj}{\bf{Conjecture}}

\newcommand{\N}{\mathbb{N}}
\newcommand{\Z}{\mathbb{Z}}

\newcommand{\tM}{\mathbf{M}}
\newcommand{\tN}{\mathbf{N}}
\newcommand{\tS}{\mathbf{S}}
\newcommand{\tT}{\mathbf{T}}

\DeclareMathOperator{\ext}{ext}

\numberwithin{equation}{section}

\def\pn{\par\noindent}

\begin{document}

\title[Conciseness on normal subgroups and new concise words]{Conciseness on normal subgroups and new concise words from lower central and derived words}

\author{Gustavo A.\ Fern\'andez-Alcober}
\address{{Department of Mathematics}, {University of the Basque Country UPV/EHU,} {Bilbao, Spain}}
\curraddr{}
\email{gustavo.fernandez@ehu.eus}
\thanks{}

\author{Matteo Pintonello}
\address{{Department of Mathematics}, {University of the Basque Country UPV/EHU,} {Bilbao, Spain}}
\curraddr{}
\email{matteo.pintonello@ehu.eus}
\thanks{}



\date{}

\dedicatory{}

\begin{abstract}
Let $w=w(x_1,\ldots,x_r)$ be a lower central word or a derived word.
We show that the word $w(u_1,\ldots,u_r)$ is concise whenever $u_1,\ldots,u_r$ are non-commutator words in disjoint sets of variables, thus proving a generalized version of a conjecture of Azevedo and Shumyatsky.
This applies in particular to words of the form $w(x_1^{n_1},\ldots,x_r^{n_r})$, where the $n_i$ are non-zero integers.
Our approach is via the study of values of $w$ on normal subgroups, and in this setting we obtain the following result: if $N_1,\ldots,N_r$ are normal subgroups of a group $G$ and the set of all
values $w(g_1,\ldots,g_r)$ with $g_i\in N_i$ is finite then also the subgroup generated by these values, i.e.\ $w(N_1,\ldots,N_r)$, is finite.
\end{abstract}

\maketitle

\section{\bf Introduction}

\vskip 0.4 true cm

Let $w=w(x_1,\ldots,x_r)$ be a group word, which we identify with an element of the free group $F=F(X)$ of countable rank on the set
$X=\{x_n\}_{n\in\N}$.
For every group $G$ and every $r$-tuple $\mathbf{g}=(g_1,\ldots,g_r)$ of elements of $G$, we can consider the element $w(\mathbf{g})$ resulting from replacing each $x_i$ with $g_i$.
We thus obtain the {\em set of word values} of $w$ in $G$,
\[
w\{G\} = \{ w(\mathbf{g}) \mid \mathbf{g} \in G\times \overset{r}{\cdots} \times G \}.
\]
The subgroup generated by $w\{G\}$ is called the {\em verbal subgroup} of $w$ in $G$, and is denoted by $w(G)$.
More generally, if $\tS=(S_1,\ldots,S_r)$ is an $r$-tuple of subsets of $G$, we can consider the set of values
\[
w\{\tS\} = \{ w(\mathbf{g}) \mid \mathbf{g} \in S_1\times \cdots \times S_r \},
\]
and the corresponding verbal subgroup on $\tS$, namely
$w(\tS)=\langle w\{\tS\} \rangle$.
Of special interest is the case when $\tS$ is a tuple $\tN=(N_1,\ldots,N_r)$ of normal subgroups of $G$.
We then say that $w(\tN)$ is the {\em $\tN$-verbal subgroup} of $w$ and that it is a verbal subgroup on normal subgroups.

If a word $w$ satisfies the property that, whenever the set of values $w\{G\}$ is finite in a group $G$, then the verbal subgroup $w(G)$ is also finite, we say that $w$ is a {\em concise} word.
Philip Hall conjectured that all words are concise.
This is one of his three famous problems on verbal subgroups; see Section 4.2 in Robinson's book
\cite{Ro} or Section 1.4 in Segal's book \cite{Se}.
In the attempt to settle this conjecture, the following types of words were proved to be concise:
\begin{enumerate}
\item
Non-commutator words, that is, words lying outside $F'$.
Equivalently, these are the words for which the exponent sum of some variable $x_i$ is non-zero.
(P.\ Hall, unpublished.)
\item
Lower commutator words $\gamma_r$, defined by
\[
\gamma_r(x_1,\ldots,x_r) = [x_1,\ldots,x_r],
\]
and whose corresponding verbal subgroup is the $r$th term of the lower central series.
(P.\ Hall, unpublished.)
\item
Derived words $\delta_k$, defined recursively by $\delta_0=x_1$ and
\[
\delta_k(x_1,\ldots,x_{2^k})
=
[\delta_{k-1}(x_1,\ldots,x_{2^{k-1}}),\delta_{k-1}(x_{2^{k-1}+1},\ldots,x_{2^k})].
\]
(Turner-Smith \cite{TS}.) 
The verbal subgroup of $\delta_k$ is the $k$th derived subgroup of the group.
\item
More generally, outer commutator words, also known as multilinear words.
(Wilson \cite{Wi}.)
These are words obtained by nesting commutators of variables arbitrarily, but never repeating a variable.
For example, $\big[[x_1,x_2,x_3],[[x_4,x_5],[x_6,x_7]]\big]$ is an outer commutator.
\end{enumerate}
For convenience, a single variable is also considered to be an outer commutator word.
Obviously, lower central words and derived words are outer commutator words.

However, Ivanov \cite{Iv} refuted P. Hall's conjecture in 1989, showing that the word
$[[x^{pn},y^{pn}]^n,y^{pn}]^n$ is not concise for big enough $n$ and $p$, where $n$ is odd and $p$ is a prime.
Another counterexample can be found in Theorem 39.7 of Olshanskii's book \cite{Ol}.

It was not until 2011 when new examples of concise words were discovered.
In two independent papers, Abdollahi and Russo \cite{Ab-Ru}, and Fern\'andez-Alcober, Morigi and Traustason \cite{Fe-Mo-Tr} proved that the Engel word
\[
e_n(x,y)=[x,y,\overset{n}{\ldots},y]
\]
is concise for $n=2$, $3$ and $4$.
To date, it is not known whether $e_n$ is concise for $n\ge 5$.
The topic of conciseness experienced a revival in the last years in the works of Shumyatsky and his collaborators.
On the one hand, they have studied some variations of conciseness in profinite groups, and on the other hand, they have proved that the following families of words are concise in all groups:
\begin{enumerate}
\setcounter{enumi}{4}
\item
If $u_1$ and $u_2$ are non-commutator words in disjoint sets of variables, then $[u_1,u_2]$ is concise.
(Delizia, Shumyatsky, Tortora, and Tota \cite{De-Sh-To-To}.)
The same happens with commutators $[u_1,u_2,u_3]$ of non-commutators.
(Azevedo, Shumyatsky \cite{Az-Sh}.)
\item
If $u=u(x_1,\ldots,x_k)$ is a non-commutator word and we produce copies of $u$ in disjoint variables by setting $u_i=u(x_{i1},\ldots,x_{ik})$ for $i=1,\ldots,r$, then $[u_1,\ldots,u_r]$ and $[v,u_1,\ldots,u_r]$ are concise, where $v$ is another non-commutator, in a set of variables disjoint from those of $u_1,\ldots,u_r$.
(Azevedo, Shumyatsky \cite{Az-Sh}.)
\item
If $u$ is an outer commutator word and $v$ is a non-commutator in disjoint variables, then $[u,v]$ is concise.
(Azevedo, Shumyatsky \cite{Az-Sh}.)
\end{enumerate}
In the light of (v) and (vi) above, Azevedo and Shumyatsky made the following conjecture in \cite{Az-Sh}.

\begin{conj}[Azevedo, Shumyatsky]
Let $u_1,\ldots,u_r$ be non-commutator words in disjoint sets of variables.
Then the word
\[
\gamma_r(u_1,\ldots,u_r) = [u_1,\ldots,u_r]
\]
is concise.
\end{conj}

The main result in this paper is the proof of this conjecture, not only for lower central words, but also for derived words.

\begin{thmA}
Let $w=w(x_1,\ldots,x_r)$ be a lower central word or a derived word.
If $u_1,\ldots,u_r$ are non-commutator words in disjoint sets of variables, then the word $w(u_1,\ldots,u_r)$ is concise.
In particular, the word $w(x_1^{n_1},\ldots,x_r^{n_r})$ is concise whenever
$n_1,\ldots,n_r\in\Z\smallsetminus\{0\}$.
\end{thmA}

In \cite{Fe-Mo} Fern\'andez-Alcober and Morigi showed that if a word $w$ is concise then it is actually {\em boundedly concise}, that is,
there is a function $f_w:\N\rightarrow \N$ such that whenever $|w\{G\}|=m$ in a group $G$, we have $|w(G)|\le f_w(m)$.
Their argument requires the use of non-principal ultrafilters over infinite sets, whose existence is independent of Zermelo-Fraenkel set theory.
We remark that our proof of Theorem A shows that the words therein are boundedly concise, without relying on the result in \cite{Fe-Mo}.

Our approach to Theorem A is via the study of $w\{\tN\}$ and $w(\tN)$ for an $r$-tuple $\tN=(N_1,\ldots,N_r)$ of normal subgroups of a group $G$.
More precisely, for every $i=1,\ldots,r$ we consider a generating set $S_i$ of $N_i$ that is a normal subset of $G$, and relate $w(\tN)$ to $w\{\tS\}$, where $\tS=(S_1,\ldots,S_r)$.
Under an additional assumption on the sets $S_i$, we are able to prove that the finiteness of
$w\{\tS\}$ implies that of $w(\tN)$ (see \cref{gamma_r concise most general} and
\cref{delta_k concise most general}).
Theorem A is a corollary of this general result, and another consequence of it is the theorem below, which we think has an independent interest.

\begin{thmB}
Let $w=w(x_1,\ldots,x_r)$ be a lower central word or a derived word.
Assume that $\tN=(N_1,\ldots,N_r)$ is a tuple of normal subgroups of a group $G$ such that
$w\{\tN\}$ is finite.
Then the subgroup $w(\tN)$ is also finite.
\end{thmB}

If a word $w$ satisfies the property stated in Theorem B, we say that $w$ is
\emph{concise on normal subgroups}.
Note that Theorem B was known for the usual commutator word $w=[x,y]$, as a consequence of Theorem 4.14 in Robinson's book \cite{Ro}.
Actually, it suffices to require that either of the subgroups $H$ or $K$ are ascendant in $\langle H, K \rangle$ in order to conclude that $[H,K]$ is finite whenever the set $\{[h,k]\mid h\in H,\ k\in K\}$ is finite.
Baer \cite{Ba} had previously proved this result under the assumption that one of $H$ or $K$ is normal in $\langle H,K \rangle$.
We impose the stronger condition that all subgroups should be normal, but in return we generalize the result to all lower central and derived words.
Theorem B actually holds with bounds: if $w\{\tN\}$ is of order $m$, then the order of $w(\tN)$ can be bounded by a function of $m$ and $w$.
We conclude this introduction with an extension of the conjecture of Azevedo and Shumyatsky.

\begin{conj}
Theorems A and B hold for an arbitrary outer commutator word $w$.
\end{conj}


\section{\bf Preliminaries and lower central words}

\vskip 0.4 true cm

We start with a few results regarding word values and verbal subgroups on normal subgroups or on normal subsets, in the case of outer commutator words.
If $w=w(x_1,\ldots,x_r)$ is an outer commutator word that is not a variable, then we can write $w=[\alpha,\beta]$, where $\alpha$ and $\beta$ are again outer commutator words.
Without loss of generality, after renaming variables if necessary, we may assume that $\alpha=\alpha(x_1,\ldots,x_q)$ and $\beta=\beta(x_{q+1},\ldots,x_r)$, with $1\le q<r$.

\begin{lem}
\label{generators ocw normal subgroups}
Let $w=w(x_1,\ldots,x_r)$ be an outer commutator word, and let
$\tN=(N_1,\ldots,N_r)$ be an $r$-tuple of normal subgroups of a group $G$.
\begin{enumerate}
\item 
Assume that $w=[\alpha,\beta]$, with $\alpha=\alpha(x_1,\ldots,x_q)$ and $\beta=\beta(x_{q+1},\ldots,x_r)$.
If we set $\tN_1=(N_1,\ldots,N_q)$ and $\tN_2=(N_{q+1},\ldots,N_r)$, then
$w(\tN)=[\alpha(\tN_1),\beta(\tN_2)]$.
\item
Assume that $N_i=\langle S_i \rangle$ for every $i=1,\ldots,r$, where each $S_i$ is a normal subset of $G$.
If we set $\tS=(S_1,\ldots,S_r)$, then the subgroup $w(\tN)$ is generated by $w\{\tS\}$.
\end{enumerate}
\end{lem}

\begin{proof}
Both (i) and (ii) follow immediately from the simple fact that if $S$ and $T$ are two normal subsets of a group $G$ then
\[
[\langle S \rangle, \langle T \rangle] = \langle [s,t] \mid s\in S,\ t\in T \rangle,
\]
where for part (ii) we use (i) and induction on the number of variables.
\end{proof}

We are interested in words of the form $w(u_1,\ldots,u_r)$, where $u_1,\ldots,u_r$ are non-commutator words that involve different variables.
Let us introduce the following concept.

\begin{defn}
Let $u_1,\ldots,u_r$ be group words.
We say that these words are {\em disjoint} if the sets of variables that they involve are pairwise disjoint.
\end{defn}

If $w$ is a word in $r$ variables and $u_1,\ldots,u_r$ are disjoint words, then the set of values of the word $w^*=w(u_1,\ldots,u_r)$ in a group $G$ can be written as $w\{\tS\}$, where
\[
\tS = (u_1\{G\},\ldots,u_r\{G\}).
\]
Since every $u_i\{G\}$ is a normal subset of $G$, we get the following consequence of the previous lemma.

\begin{cor}
\label{verbal subgroup of ocw of other words}
Let $w=w(x_1,\ldots,x_r)$ be an outer commutator word and let $u_1,\ldots,u_r$ be arbitrary disjoint words.
If $w^*=w(u_1,\ldots,u_r)$ then for every group $G$ we have
\[
w^*(G) = w(u_1(G),\ldots,u_r(G)).
\]
\end{cor}

Now we want to make part (ii) of \cref{generators ocw normal subgroups} quantitative.
If we take a standard generator $w(n_1,\ldots,n_r)$ of $w(\tN)$, with $n_i\in N_i$, how can we estimate the number of factors from $w\{\tS\}^{\pm 1}$ that are needed to write it?
We need to introduce the following notation.

\begin{defn}
\label{ast product}
Let $G$ be a group and let $S$ be a subset of $G$.
For every $n\in\N$, we define $S^{\ast n}$ to be the set of all products of elements of $S\cup S^{-1}$ of length at most $n$.
\end{defn}

In other references, $S^{\ast n}$ is defined as the set of products of exactly $n$ elements of $S$.
Since we can always replace $S$ with $S\cup S^{-1}\cup \{1\}$, both definitions are basically equivalent.
We prefer the definition above because it suits better the description of the elements of the subgroup $\langle S \rangle$.
Also, with this definition, we have $S^{\ast k}\subseteq S^{\ast n}$ whenever $n\ge k$.
Note that if $|S|\le m$ then $|S^{\ast n}|\le (2m+1)^n$ for every $n\in\N$.
Let us connect this concept with values of outer commutator words.

\begin{lem}
\label{one value in ocw}
Let $w=w(x_1,\ldots,x_r)$ be an outer commutator word, and let $S$ be a normal subset of a group $G$.
Suppose that $\mathbf{t}=(t_1,\ldots,t_r)$ is a tuple of elements of $G$, one of whose components belongs to $S$.
Then $w(\mathbf{t})\in S^{\ast 2^{r-1}}$.
\end{lem}

\begin{proof}
The result is obvious for $r=1$, so we assume $r>1$.
Then we can write $w(\mathbf{t})=[\alpha(\mathbf{t'}),\beta(\mathbf{t''})]$, where $\alpha$ and $\beta$ are outer commutator words, and the tuples $\mathbf{t'}$ and $\mathbf{t''}$ form a partition of $\mathbf{t}$.
Assume without loss of generality that $\mathbf{t'}$ contains an entry from $S$.
By induction on $r$, we have $\alpha(\mathbf{t'})\in S^{\ast 2^{r-2}}$.
Consequently,
\[
w(\mathbf{t}) = \alpha(\mathbf{t'})^{-1} \alpha(\mathbf{t'})^{\beta(\mathbf{t''})} \in S^{\ast 2^{r-1}},
\]
since $S$ is a normal subset of $G$.
\end{proof}

On the other hand, by Lemma 2.8 of \cite{dlH-Mo}, if $w=w(x_1,\ldots,x_r)$ is an outer commutator and $g_1,\ldots,g_r,h$ are elements of a group $G$, then for every $i=1,\ldots,r$ we have
\[
w(g_1,\ldots,g_{i-1},g_i h,g_{i+1},\ldots,g_r)
=
w(g_1^*,\ldots,g_{i-1}^*,g_i^*,g_{i+1}^*,\ldots,g_r^*)
\, \cdot
w(g_1,\ldots,g_{i-1},h,g_{i+1},\ldots,g_r),
\]
where $g_j^*$ is a conjugate of $g_j$ in $G$ for every $j=1,\ldots,r$.
The following lemma follows easily from this result by induction on
$m_1\cdots m_r$.

\begin{lem}
\label{width ocw}
Let $w=w(x_1,\ldots,x_r)$ be an outer commutator word, and let $\tS=(S_1,\ldots,S_r)$ be a tuple of normal subsets of a group $G$.
If $\mathbf{t}=(t_1,\ldots,t_r)$ with $t_i\in S_i^{\ast m_i}$ for every $i=1,\ldots,r$, then
\[
w(\mathbf{t}) \in w\{\tS\}^{\ast m_1\dots m_r}.
\]
\end{lem}

In an abelian group $G$, the word map $(g_1,\ldots,g_r)\mapsto w(g_1,\ldots,g_r)$ is a group homomorphism for every word $w$, and consequently $w(G)=w\{G\}$.
Of course, this is far from being true in arbitrary groups.
Outer commutator words, although they are also called multilinear words because the same type of commutator arrangements yields multinear words in Lie rings, are not multilinear in groups.
However, our approach to proving Theorems A and B relies on showing that, in suitable sections that cover the section $w(\tN)/w(\tN)'$, outer commutator words are linear in one specific variable (which depends on the section).
To this purpose, we give the following definition.

\begin{defn}
Let $w=w(x_1,\ldots,x_r)$ be a word and let $\tN=(N_1,\ldots,N_r)$ be an $r$-tuple of normal subgroups of a group $G$.
We say that $w$ is {\em linear in position $i$ of the tuple $\tN$}
provided that,
for all $g_j\in N_j$ for $j=1,\ldots,r$ and $h_i\in N_i$, we have
\begin{equation}
\label{def linearity}
w(g_1,\ldots,g_{i-1},g_ih_i,g_{i+1},\ldots,g_r)
=
w(g_1,\ldots,g_{i-1},g_i,g_{i+1},\ldots,g_r)
\, \cdot
w(g_1,\ldots,g_{i-1},h_i,g_{i+1},\ldots,g_r).
\end{equation}
\end{defn}

Typically, we will search for linearity in a normal section $K/L$ of the ambient group $G$ that is generated by the image of $w\{\tN\}$, so that condition \eqref{def linearity} above is required to hold modulo $L$.
Obviously, this type of linearity is inherited by sections of the form $KN/LN$ for a given $N\trianglelefteq G$.
Next we show that it is also preserved under taking suitable commutators.

\begin{lem}
\label{linearity after commutator}
Let $w=[\alpha,\beta]$ be an outer commutator word, where
$\alpha=\alpha(x_1,\ldots,x_q)$ and $\beta=\beta(x_{q+1},\ldots,x_r)$.
Assume that $\tN=(N_1,\ldots,N_r)$ is a tuple of normal subgroups of a group $G$, and set $\tN_1=(N_1,\ldots,N_q)$ and
$\tN_2=(N_{q+1},\ldots,N_r)$.
Then the following hold:
\begin{enumerate}
\item
If $K/L$ is a normal section of $G$ generated by the image of $\alpha\{\tN_1\}$ and
$\alpha$ is linear in component $i$ of $\tN_1$ modulo $L$, then
the section $U/V$, where $U=[K,\beta(\tN_2)]$ and
$V=[w(\tN),\alpha(\tN_1)][L,\beta(\tN_2)]$, is generated by the image of $w\{\tN\}$ and $w$ is linear in component $i$ of $\tN$ modulo $V$.
\item
If $K/L$ is a normal section of $G$ generated by the image of $\beta\{\tN_2\}$ and
$\beta$ is linear in component $i$ of $\tN_2$ modulo $L$, then
the section $U/V$, with $U=[\alpha(\tN_1),K]$ and
$V=[w(\tN),\beta(\tN_2)][\alpha(\tN_1),L]$, is generated by the image of $w\{\tN\}$ and $w$ is linear in component $q+i$ of $\tN$ modulo $V$.
\end{enumerate}
\end{lem}

\begin{proof}
We only prove part (i).
To start with, we have
\[
[K,\beta(\tN_2)]
=
[\alpha(\tN_1)L,\beta(\tN_2)]
=
[\alpha(\tN_1),\beta(\tN_2)] \, [L,\beta(\tN_2)]
=
w(\tN) \, [L,\beta(\tN_2)],
\]
where the last equality follows from (i) of \cref{generators ocw normal subgroups}.
Thus the section $U/V$ is generated by the image of $w\{\tN\}$.

As for the assertion about linearity, let us consider the general congruence stating that $\alpha$ is linear in component $i$ of $\tN_1$ modulo $L$.
This can be written in the form $x\equiv yz \pmod L$, where $x$, $y$, and $z$ are like the three elements appearing in \eqref{def linearity} (with $\alpha$ playing the role of $w$).
In particular, $x,y,z\in\alpha\{\tN_1\}$.
Standard commutator identities then yield that, for every $n\in\beta\{\tN_2\}$, we have
\[
[x,n] \equiv [y,n][z,n]
\pmod{[\alpha(\tN_1),\beta(\tN_2),\alpha(\tN_1)] \, [L,\beta(\tN_2)]}.
\]
This proves the result.
\end{proof}

We can use the previous lemma to determine, for every lower central word $\gamma_r$ and every $r$-tuple $\tN$ of normal subgroups, a series from
$[\gamma_r(\tN),\gamma_r(\tN)]$ to $\gamma_r(\tN)$ that is linear for $\gamma_r$ at every section.

\begin{thm}
\label{linear series gamma}
Let $r\in\N$.
Assume that $\tN=(N_1,\ldots,N_r)$ is a tuple of normal subgroups of a group $G$,
and define
\[
\tN_i=(N_1,\ldots,N_{i-1},\gamma_i(N_1,\ldots,N_i),N_{i+1},\ldots,N_r)
\]
for every $i=1,\ldots,r$.
Then there is a series
\[
[\gamma_r(\tN),\gamma_r(\tN)] = P_{r+1}^r
\le
P_r^r
\le
\cdots
\le
P_i^r
\le
\cdots
\le
P_1^r = \gamma_r(\tN)
\]
such that, for every $i=1,\ldots,r$, the section $P_i^r/P_{i+1}^r$ is generated by the image of $\gamma_r\{\tN_i\}$ and the word $\gamma_r$ is linear in component $i$ of $\tN_i$ modulo $P_{i+1}^r$.
\end{thm}

\begin{proof}
Set $Q_i^r=\gamma_r(\tN_i)$ for every $i=1,\ldots,r$, and
$Q_{r+1}^r=[\gamma_r(\tN),\gamma_r(\tN)]$.
Obviously, the conditions that $P_{r+1}^r=[\gamma_r(\tN),\gamma_r(\tN)]$ and that
$P_i^r/P_{i+1}^r$ is generated by the image of $Q_i^r$ mean that we need to choose
$P_i^r = Q_i^r Q_{i+1}^r \ldots Q_{r+1}^r$, for $i=1,\ldots,r+1$.

Let us then prove the linearity of $\gamma_r$ in component $i$ of $\gamma_r(\tN_i)$ modulo $P_{i+1}^r$.
We argue by induction on $r-i$.
The basis of the induction, $i=r$, follows from the congruence
$[g,x_ry_r]\equiv [g,x_r][g,y_r] \pmod{P_{r+1}}$ for all $g\in G$ and
$x_r,y_r\in\gamma_r(\tN)$,
which holds because $P_{r+1}^r=[\gamma_r(\tN),\gamma_r(\tN)]$.

Let us now assume that $1\le i<r$ and that the result is true for differences less than $r-i$.
For every $i=1,\ldots,r$, let $Q_i^{r-1}$ and $P_i^{r-1}$ be defined from the tuple $\tN^*=(N_1,\ldots,N_{r-1})$ in the same way as we defined $P_i^r$ and $Q_i^r$ from $\tN$.
Then linearity holds in position $i$ of 
\[
\tN_i^* = (N_1,\ldots,N_{i-1},\gamma_i(N_1,\ldots,N_i),N_{i+1},\ldots,N_{r-1})
\]
modulo $P_{i+1}^{r-1}$.
Now we apply \cref{linearity after commutator} by taking $K=P_i^{r-1}$, $L=P_{i+1}^{r-1}$,
$\alpha=\gamma_{r-1}$ and $\beta=x_r$.
Thus $\gamma_r$ is linear in component $i$ of $\tN_i$ modulo the subgroup
\begin{equation}
\label{modulus gamma}
[\gamma_r(\tN),\gamma_{r-1}(\tN^*)] \, [P_{i+1}^{r-1},N_r].
\end{equation}
Observe that
\[
[\gamma_r(\tN),\gamma_{r-1}(\tN^*)] = 
[N_1,\ldots,N_{r-1},\gamma_r(N_1,\ldots,N_r)] = Q_r^r \le P_{i+1}^r,
\]
since $r-i\ge 1$.
On the other hand,
\[
[P_{i+1}^{r-1},N_r] =  \Bigg(\prod_{j=i+1}^{r-1} \, [Q_j^{r-1},N_r]\Bigg) \cdot [Q_r^{r-1},N_r] = \Bigg( \prod_{j=i+1}^{r-1} \, Q_j^r \Bigg)\cdot [Q_r^{r-1},N_r],
\]
and
\[
[Q_r^{r-1},N_r]
=
[\gamma_{r-1}(N_1,\ldots,N_{r-1}),\gamma_{r-1}(N_1,\ldots,N_{r-1}),N_r]
\le [N_1,\ldots,N_{r-1},\gamma_r(N_1,\ldots,N_r)] = Q_r^r,
\]
where the inclusion follows from P.\ Hall's Three Subgroup Lemma.
Hence the subgroup in \eqref{modulus gamma} is contained in $P_{i+1}^r$, and the result follows.
\end{proof}

We are now in a position to prove the key theorem that will provide both Theorem A and Theorem B for lower central words.
We need the following version for normal subgroups of a well-known lemma in the theory of concise words (see, for example, \cite[Lemma 4]{De-Mo-Sh}).
The proof is exactly the same, based on Schur's Theorem, and taking into account also part (ii) of \cref{generators ocw normal subgroups} in this case, so we omit it.
In the remainder of the paper, for a tuple $S$ of parameters, we use the expression {\em $S$-bounded} to mean ``bounded by a function of $S$".

\begin{lem}
\label{w(N)' finite}
Let $w=w(x_1,\ldots,x_r)$ be an arbitrary word and let $\tN=(N_1,\ldots,N_r)$ be an $r$-tuple of normal subgroups of a group $G$.
Suppose that $N_i=\langle S_i \rangle$, where $S_i$ is a normal subset of $G$ for every $i=1,\ldots,r$, and set $\tS=(S_1,\ldots,S_r)$.
If $w\{\tS\}$ is finite of order $m$ then $w(\tN)'$ is finite of $m$-bounded order.
\end{lem}

\begin{thm}
\label{gamma_r concise most general}
Let $r\in\N$ and let $\tN=(N_1,\ldots,N_r)$ be a tuple of normal subgroups of a group $G$.
Assume that $N_i=\langle S_i \rangle$ for every $i=1,\ldots,r$, where:
\begin{enumerate}
\item
$S_i$ is a normal subset of $G$.
\item
There exists $n_i\in\N$ such that all $n_i$th powers of elements of $N_i$ are contained in $S_i$.
\end{enumerate}
If  for the tuple $\tS=(S_1,\ldots,S_r)$ the set of values $\gamma_r\{\tS\}$ is finite of order $m$, then the subgroup
$\gamma_r(\tN)$ is also finite, of $(m,r,n_1,\ldots,n_r)$-bounded order.
\end{thm}

\begin{proof}
We follow the notation $\tN_i$ and $P_i^r$, introduced in the statement of
\cref{linear series gamma}.
We are going to prove that $P_i^r$ is finite of bounded order for $i=1,\ldots,r+1$ by reverse induction on $i$.
Since $P_1^r=\gamma_r(\tN)$, this proves the result.

The basis of the induction follows from \cref{w(N)' finite}, since $P_{r+1}^r=[\gamma_r(\tN),\gamma_r(\tN)]$.
Let us assume that $P_{i+1}^r$ is finite of bounded order and prove that the same holds for $P_i^r$.
Recall that the quotient $P_i^r/P_{i+1}^r$ is the image of $\gamma_r(\tN_i)$, and
then, by a suitable application of \cref{generators ocw normal subgroups}, it can be generated by the images of the set $\tT$ of commutators
\[
[s_1,\ldots,s_{i-1},x_i,s_{i+1},\ldots,s_r],
\]
with $s_j\in S_j$ for $1\le j\le r$, $j\ne i$, and $x_i\in \gamma_i\{\tS_i\}$, where $\tS_i=(S_1,\ldots,S_i)$.
By \cref{one value in ocw}, we have $\gamma_i\{\tS_i\}\subseteq S_i^{\ast 2^{i-1}}$, and then \cref{width ocw} implies that
\[
[s_1,\ldots,s_{i-1},x_i,s_{i+1},\ldots,s_r]\in \gamma_r\{\tS\}^{\ast 2^{i-1}} \subseteq \gamma_r\{\tS\}^{\ast 2^{r-1}}.
\]
From the assumption that $|\gamma_r\{\tS\}|=m$, we get
\[
|\tT| \le (2m+1)^{2^{r-1}},
\]
and consequently $P_i^r/P_{i+1}^r$ can be generated by an $(m,r)$-bounded number of elements.
Since $P_i^r/P_{i+1}^r$ is abelian, the proof will be complete once we show that all elements in $\tT$ have bounded finite order modulo $P_{i+1}^r$.

By \cref{linear series gamma}, the word $\gamma_r$ is linear in position $i$ of the tuple $\tN_i$ modulo $P_{i+1}^r$.
In particular,
\begin{equation}
\label{power out}
[s_1,\ldots,s_{i-1},x_i,s_{i+1},\ldots,s_r]^{\lambda n_i}
\equiv
[s_1,\ldots,s_{i-1},x_i^{\lambda n_i},s_{i+1},\ldots,s_r]
\pmod{P_{i+1}^r},
\end{equation}
for every $\lambda\in\Z$.
Since $x_i\in\gamma_i(N_1,\ldots,N_i)\le N_i$, it follows from (ii) in the statement of the theorem that $x_i^{\lambda n_i}\in S_i$ for all $\lambda\in\Z$.
Thus we get
\[
[s_1,\ldots,s_{i-1},x_i^{\lambda n_i},s_{i+1},\ldots,s_r] \in \gamma_r\{\tS\}.
\]
Since $\gamma_r\{\tS\}$ is finite of order $m$, it follows that there exist $\lambda,\mu\in\{0,\ldots,m\}$, $\lambda\ne\mu$, such that
\[
[s_1,\ldots,s_{i-1},x_i,s_{i+1},\ldots,s_r]^{\lambda n_i}
\equiv
[s_1,\ldots,s_{i-1},x_i,s_{i+1},\ldots,s_r]^{\mu n_i}
\pmod{P_{i+1}^r}.
\]
This implies that $[s_1,\ldots,s_{i-1},x_i,s_{i+1},\ldots,s_r]$ has $(m,n_i)$-bounded finite order modulo $P_{i+1}^r$, as desired. 
\end{proof}

If we take $S_i=N_i$, we get Theorem B for the lower central words.

\begin{cor}
\label{gamma_r concise on normal}
Let $r\in\N$ and let $\tN=(N_1,\ldots,N_r)$ be a tuple of normal subgroups of a group $G$.
If $\gamma_r\{\tN\}$ is finite of order $m$, then the subgroup $\gamma_r(\tN)$ is also finite, of $(m,r)$-bounded order.
\end{cor}

Now we deduce Theorem A for lower central words.

\begin{cor}
\label{gamma_r of non-comm concise}
Let $r\in\N$ and let $u_1,\ldots,u_r$ be disjoint non-commutator words.
Then the word $\gamma_r(u_1,\ldots,u_r)$ is boundedly concise.
In particular, $\gamma_r(x_1^{n_1},\ldots,x_r^{n_r})$ is boundedly concise for all $n_i\in \Z\smallsetminus\{0\}$.
\end{cor}

\begin{proof}
Let us consider the word $w=\gamma_r(u_1,\ldots,u_r)$, and
let $G$ be a group in which $w$ takes finitely many values, say $|w\{G\}|=m$.
By \cref{verbal subgroup of ocw of other words}, we have $w(G)=\gamma_r(u_1(G),\ldots,u_r(G))$.
Note that $u_i(G)=\langle S_i \rangle$, where $S_i=u_i\{G\}$, and that
$w\{G\}=\gamma_r\{\tS\}$, where $\tS=(S_1,\ldots,S_r)$.
Now observe that $S_i$ is a normal subset of $G$ and that, since $u_i$ is a non-commutator word, for some $n_i\in\Z\smallsetminus\{0\}$ we have $\{g^{n_i}\mid g\in G\}\subseteq u_i\{G\}$.
Hence $w(G)$ is finite of $(m,r,n_1,\ldots,n_r)$-bounded order by 
\cref{gamma_r concise most general}.
\end{proof}


\vskip 0.8 true cm

\section{\bf Derived words}

\vskip 0.4 true cm

In this section, we prove Theorems A and B for derived words.
The general strategy is still the same as for lower central words: we are going to obtain a suitable series of normal subgroups of $G$, going from
$[\delta_k(\tN),\delta_k(\tN)]$ to $\delta_k(\tN)$, with the property that each of the factors of the series can be generated by a verbal subgroup on a tuple of normal subgroups that is closely related to $\delta_k(\tN)$ and linear in one component.
This is basically \cref{linear series delta} below.
For simplicity, let us refer to such a series as a linear series.
The argument needed to obtain a linear series for derived words presents difficulties and subtleties that did not arise with lower central words, and is also significantly more technical.
For the convenience of the reader and in order to make the procedure for a general $\delta_k$ more understandable, first of all we are going to provide a sketch of it in the particular case of $\delta_2$.

Of course, $\delta_1=\gamma_2$ and, according to \cref{linear series gamma}, we have the following linear series
for $\delta_1(N_1,N_2)$:
\begin{center}
\begin{tikzcd}
{[\tcbhighmath[drop fuzzy shadow]{N_1},N_2]}
\arrow[d, dash]
\\
{\big[N_1,\tcbhighmath[drop fuzzy shadow]{[N_1,N_2]}\big]} 
\arrow[d, dash]
\\
{\big[[N_1,N_2],[N_1,N_2]\big]}
\\
\end{tikzcd}
\captionof{figure}{Series of $[N_1,N_2]$}
\end{center}
In this and in the next diagrams, a red box indicates the component in which we have linearity.

Let us see how we can construct a linear series for $\delta_2(N_1,N_2,N_3,N_4)$ from the series above for $\delta_1$.
To this purpose, we will use \cref{linearity after commutator}, which ensures that linearity is preserved after taking suitable commutators, and also the remark made before that lemma, saying that linearity is preserved after multiplying by a normal subgroup. 
To start with, we take the commutator of the terms of the previous series with $[N_3,N_4]$, obtaining the series
\begin{center}
\begin{tikzcd}
{\big[[\tcbhighmath[drop fuzzy shadow]{N_1},N_2],[N_3,N_4]\big]}
\arrow[d, dash]\\
{\Big[\big[N_1,\tcbhighmath[drop fuzzy shadow]{[N_1,N_2]}\big],[N_3,N_4]\Big]}
\arrow[d, dash]\\
{\Big[ \big[[N_1,N_2],[N_1,N_2] \big],[N_3,N_4] \Big]}
\\
\end{tikzcd}
\end{center}
Now we multiply this series with $\big[ [N_1,N_2],[[N_1,N_2],[N_3,N_4]] \big]$, that contains $\big[ [N_1,N_2],[N_1,N_2],[N_3,N_4] \big]$ by P.\ Hall's Three Subgroup Lemma, and we get the following diagram:
\begin{center}
\begin{tikzcd}
{\big[[\tcbhighmath[drop fuzzy shadow]{N_1},N_2],[N_3,N_4]\big]}
\arrow[d, dash]\\
{\Big[\big[N_1,\tcbhighmath[drop fuzzy shadow]{[N_1,N_2]}\big],[N_3,N_4]\Big]}
\arrow[d, dash]\\
{\Big[[N_1,N_2],\big[[N_1,N_2],[N_3,N_4]\big]\Big]}
\\
\end{tikzcd}
\captionof{figure}{First diagram for $\big[[N_1,N_2],[N_3,N_4]\big]$}
\end{center}
Here, and in the remaining diagrams, instead of the subgroups of the series, we are showing verbal subgroups on normal subgroups whose images coincide with the corresponding factors of the series.
After all, it is in these subgroups where we are going to obtain the linearity conditions.
Be aware then that vertical lines in the diagrams do not denote inclusions from this point onwards.

By swapping the roles of $(N_1,N_2)$ and $(N_3,N_4)$, we can obtain this 
other diagram:
\begin{center}
\begin{tikzcd}
{\big[[N_1,N_2],[\tcbhighmath[drop fuzzy shadow]{N_3},N_4]\big]}
\arrow[d, dash]\\
{\Big[[N_1,N_2],\big[N_3,\tcbhighmath[drop fuzzy shadow]{[N_3,N_4]}\big]\Big]}
\arrow[d, dash]\\
{\Big[\big[[N_1,N_2],[N_3,N_4]\big],[N_3,N_4]\Big]}
\\
\end{tikzcd}
\vspace{-5pt}
\captionof{figure}{Second diagram for $\big[[N_1,N_2],[N_3,N_4]\big]$}
\end{center}
Now we take the commutator of $[N_1,N_2]$ with the terms of this last diagram,
and we add the extra term $\delta_2(N_1,N_2,N_3,N_4)'$ at the bottom:
\begin{center}
\begin{tikzcd}
{\Big[[N_1,N_2],\big[[N_1,N_2],[\tcbhighmath[drop fuzzy shadow]{N_3},N_4]\big]\Big]}
\arrow[d, dash]\\
{\Big[[N_1,N_2],\big[[N_1,N_2],\big[N_3,\tcbhighmath[drop fuzzy shadow]{[N_3,N_4]}\big]\big]\Big]}
\arrow[d, dash]\\
{\Big[[N_1,N_2],\big[\tcbhighmath[drop fuzzy shadow]{\big[[N_1,N_2],[N_3,N_4]\big]},[N_3,N_4]\big]\Big]}
\arrow[d, dash]\\
{\Big[\big[[N_1,N_2],[N_3,N_4]\big],\big[[N_1,N_2],[N_3,N_4]\big]\Big]}
\\
\end{tikzcd}
\vspace{-5pt}
\captionof{figure}{Series of $\Big[[N_1,N_2],\big[[N_1,N_2],[N_3,N_4]\big]\Big]$}
\end{center}
Finally, by gluing the diagrams in Figures 2 and 4 together, we obtain a linear series for the subgroup $\delta_2(N_1,N_2,N_3,N_4)$.

Of course, this is simply a sketch without proofs, but we are going to follow the same procedure in the proof of
\cref{linear series delta}, in order to get a linear series for
$\delta_k$ from a series for $\delta_{k-1}$.
At this point, it is worth noting an important difference with the situation for a lower central word $\gamma_r$.
In that case, every factor of the linear series is of the following form
(again we show the linear component in red):
\[
\big[N_1,\ldots,N_{i-1},\tcbhighmath[drop fuzzy shadow]{[N_1,\ldots,N_i]},N_{i+1},\ldots,N_r\big].
\]
We observe that this subgroup is of the form $\gamma_r(\tM)$, where the $j$th component $M_j$ of $\tM$ is either $N_j$ or a commutator of the terms of $\tN$ that involves $N_j$, and the linearity happens in $M_i$.
However, if we look at the series for $\delta_2$ obtained above, the first two subgroups in Figure 4 are
\begin{equation}
\label{first problematic}
\delta_2(N_1;N_2;[N_1,N_2];[\tcbhighmath[drop fuzzy shadow]{N_3},N_4])
\end{equation}
and
\begin{equation}
\label{second problematic}
\delta_2(N_1;N_2;[N_1,N_2];[N_3,\tcbhighmath[drop fuzzy shadow]{[N_3,N_4]}]),
\end{equation}
which are not of the form $\delta_2(\tM)$ with every $M_j$ a commutator from
$\tN$ involving $N_j$, as we can see by looking at the third component of $\delta_2$.
Also, the linearity does not happen in a component of $\delta_2$, but in a more interior position.
Nevertheless, we can write these subgroups as verbal subgroups on normal subgroups for outer commutator words different from $\delta_2$.
More specifically, if
\[
v(x_1,x_2,x_3,x_4,y_1,y_2) = \big[[x_1,x_2],[[y_1,y_2],[x_3,x_4]]\big]
\]
then the subgroups in \eqref{first problematic} and \eqref{second problematic} are $v(\tM_1)$ and $v(\tM_2)$, where
\begin{equation}
\label{longer tuples}
\tM_1=(N_1,N_2,\tcbhighmath[drop fuzzy shadow]{N_3},N_4,N_1,N_2)
\quad
\text{and}
\quad
\tM_2=(N_1,N_2,N_3,\tcbhighmath[drop fuzzy shadow]{[N_3,N_4]},N_1,N_2),
\end{equation}
where again we have marked the linear components in red.

\vskip 0.2 true cm

After having illustrated the procedure with the case of $\delta_2$, let us proceed to systematically develop the tools that are necessary for the proof of
\cref{linear series delta}.

We start by introducing a special type of words that we can derive from a given outer commutator word $w$, which we call extended words
of $w$.
Before giving the definition, we show the idea behind extended words with an example.
Consider the word $\delta_2=[[x_1,x_2],[x_3,x_4]]$.
This is formed by taking the commutator of $x_1$ and $x_2$, taking the commutator of $x_3$ and $x_4$, and then taking the commutator of these two commutators.
Now suppose that on some occasions, before performing one of these commutators, we introduce a change by taking first the commutator of one (or both) of the components with an outer commutator word not involving the variables $x_1,\ldots,x_4$ appearing in $\delta_2$.
For example, before producing $[x_1,x_2]$, we take the commutator $[[y_1,y_2],x_1]$ and now we follow as in $\delta_2$ taking the commutator with $x_2$,
obtaining $\big[ [[y_1,y_2],x_1], x_2 \big]$.
We could continue with the process of taking commutators without making any other changes, so getting
\[
\big[ \big[ [[y_1,y_2],x_1], x_2 \big], [x_3,x_4] \big],
\]
but we could also make some similar changes in the process, as in the words
\[
\Big[ \big[ [[y_1,y_2],x_1], x_2 \big], \big[x_3,[y_3,x_4]\big] \Big]
\]
and
\[
\Big[ \, \big[ [[y_1,y_2],x_1], x_2 \big], \big[ \, \big[x_3,[y_3,x_4]\big], y_4 \big] \, \Big].
\]
Another possibility is to make a commutator at the very end, after having completed $\delta_2$, as in
\[
\big[y_1,[[x_1,x_2],[x_3,x_4]]\big].
\]
Observe that all these extended words are again outer commutator words, because we never repeat a variable when we make changes in the construction of $\delta_2$.

Let us now give the formal definition of extended words. Notice that this definition differs from the one of extensions of outer commutators words given in Definition 3.1 of \cite{Fe-Sh}.

\begin{defn}
\label{def extended word}
Let $w=w(x_1,\ldots,x_r)$ be an outer commutator, and let $Y=\{y_n\}_{n\in\N}$ be a set of variables that are disjoint from $X$.
For every $k\in\N\cup \{0\}$, we define recursively the set $\ext_k(w)$ of $k$th extended words of $w$ as follows:
\begin{enumerate}
\item
$\ext_0(w)=\{w\}$.
\item
For $k\ge 1$, $\ext_k(w)$ consists of the set
\begin{multline*}
\{ [p,q], [q,p] \mid \text{$p$ outer commutator in $Y$, $q\in\ext_{k-1}(w)$, $p$ and $q$ disjoint} \}
\\
=
\{ [p,q] \mid \text{$p$ outer commutator in $Y$, $q\in\ext_{k-1}(w)$, $p$ and $q$ disjoint} \}^{\pm 1},
\end{multline*}
and, if $w=[\alpha,\beta]$, also of the set
\[
{\textstyle \bigcup\limits_{\ell+m=k}} \, \{ [p,q] \mid \text{$p\in\ext_{\ell}(\alpha)$, $q\in\ext_{m}(\beta)$, $p$ and $q$ disjoint} \}.
\]
\end{enumerate}
If $v\in\ext_k(w)$ then we say that $w$ is an {\em extended word} of degree $k$ of $w$ by outer commutators.
\end{defn}

For brevity, in the remainder we will simply speak of extended words when we mean extended words by outer commutators.
Observe that an extended word $v$ of an outer commutator $w=w(x_1,\ldots,x_r)$ is again an outer commutator, in the variables $\{x_1,\ldots,x_r\}\cup Y$.
Whenever it is convenient we will assume, after renaming the variables, that $v=v(x_1,\ldots,x_r,y_{r+1},\ldots,y_s)$.

Next we generalize \cref{width ocw} to extended words of an outer commutator word.

\begin{lem}
\label{width extended}
Let $v=v(x_1,\ldots,x_r,y_{r+1},\ldots,y_s)$ be an extended word of degree $k$ of an outer commutator word $w=w(x_1,\ldots,x_r)$.
Assume that $\tS=(S_1,\ldots,S_r)$ is a tuple of normal subsets of a group $G$.
If $\mathbf{t}=(t_1,\ldots,t_s)$ is a tuple of elements of $G$ such that $t_i\in S_i^{\ast m_i}$ for every $i=1,\ldots,r$, then
\[
v(\mathbf{t})\in w\{\tS\}^{\ast m_1\ldots m_r 2^k}.
\]
\end{lem}

\begin{proof}
We use induction on $k+r$.
If $k=0$ then $v=w$ and the result is \cref{width ocw}.
This gives in particular the basis of the induction.
Suppose now that the result holds for smaller values of $k+r$, and that $k\ge 1$.
According to \cref{def extended word}, we may assume that $v(\mathbf{t})=[p(\mathbf{t'}),q(\mathbf{t''})]$, where $p$ and $q$ are disjoint and
\begin{enumerate}
\item
either $p$ is an outer commutator word in $Y$ and $q\in\ext_{k-1}(w)$,
\item
or $p\in\ext_{\ell}(\alpha)$, $q\in\ext_{m}(\beta)$, with $w=[\alpha,\beta]$ and $\ell+m=k$.
\end{enumerate}

In case (i), all elements $t_1,\ldots,t_r$ appear in the vector $\mathbf{t''}$, and by the induction hypothesis we have
$q(\mathbf{t''})\in w\{\tS\}^{\ast m_1\ldots m_r 2^{k-1}}$.
Then the result follows by applying \cref{one value in ocw} to the commutator word $[x_1,x_2]$ and the normal subset $w\{\tS\}^{\ast m_1\ldots m_r 2^{k-1}}$.

Suppose now that we are in case (ii), and assume without loss of generality that $\alpha=\alpha(x_1,\ldots,x_q)$ and
$\beta=\beta(x_{q+1},\ldots,x_r)$.
Set $\mathbf{S'}=(S_1,\ldots,S_q)$ and $\mathbf{S''}=(S_{q+1},\ldots,S_r)$.
Since $\alpha$ and $\beta$ involve less variables than $w$, the result is true for $p$ and $q$,
and so
\[
p(\mathbf{t'})\in \alpha\{\mathbf{S'}\}^{\ast m_1\ldots m_q 2^{\ell}}
\quad
\text{and}
\quad
q(\mathbf{t''})\in \beta\{\mathbf{S''}\}^{\ast m_{q+1}\ldots m_r 2^{m}}.
\]
Now the result follows by applying \cref{width ocw} to the commutator word $[x_1,x_2]$ and the pair of normal subsets
$(\alpha\{\mathbf{S'}\},\beta\{\mathbf{S''}\})$.
\end{proof}

We also need to define a type of extensions of tuples of normal subgroups and of verbal subgroups on normal subgroups.
The idea behind the definition is to be able to deal with tuples like the ones appearing in \eqref{longer tuples} and with the corresponding verbal subgroups on normal subgroups in that paragraph.

\begin{defn}
\label{oc extension tuple}
Let $G$ be a group and let $\tN=(N_1,\ldots,N_r)$ and $\tM=(M_1,\ldots,M_s)$ be two tuples of normal subgroups of $G$.
We say that $\tM$ is an {\em outer commutator extension} of $\tN$ if the following conditions hold:
\begin{enumerate}
\item
$s\ge r$.
\item
For every $i=1,\ldots,s$, we have $M_i=w_i(\tN_i)$, where $w_i$ is an outer commutator word and all components of $\tN_i$ belong to $\tN$.
\item
For every $i=1,\ldots,r$, the subgroup $N_i$ is a component of $\tN_i$, and consequently $M_i\le N_i$.
\end{enumerate}
\end{defn}

\begin{defn}
\label{oc extension verbal}
Let $w=w(x_1,\ldots,x_r)$ be a word and let $\tN$ be an $r$-tuple of normal subgroups of a group $G$.
An {\em extension} of degree $k$ of $w(\tN)$ by outer commutators is a subgroup of the form $v(\tM)$, where $v$ is an extended word of degree $k$ of $w$ and $\tM$ is an outer commutator extension of $\tN$.
\end{defn}

For example, the subgroup in \eqref{second problematic} can be seen as an extension of $\delta_2(N_1,N_2,N_3,N_4)$ by taking
$v=\big[[x_1,x_2],[[y_1,y_2],[x_3,x_4]]\big]$ and $\tM=(N_1,N_2,N_3,[N_3,N_4],N_1,N_2)$.
Note that $v(\tM)$ is linear in the fourth component modulo the subgroup that appears below it in Figure 4.

We now prove the existence of a linear series for derived words.

\begin{thm}
\label{linear series delta}
Let $k\in \N$ and let $\tN=(N_1,\ldots,N_{2^k})$ be a tuple of normal subgroups of a group $G$.
Then there exists a series
\[
[\delta_k(\tN),\delta_k(\tN)] = V_0^k\le V_1^k\le \cdots \le V_t^k = \delta_k(\tN)
\]
of normal subgroups of $G$ of length $t=2^k+2^{k-1}-1$ such that, for every
$i=1,\ldots,t$, the following hold:
\begin{enumerate}
\item
The section $V_i^k/V_{i-1}^k$ is the image of an extension $v_i^k(\tM_i^k)$ of $\delta_k(\tN)$ of degree at most $k-1$.
\item
In the section $V_i^k/V_{i-1}^k$, the word $v_i^k$ is linear in one component of the tuple $\tM_i^k$.
\end{enumerate}
Furthermore, the words $v_i^k$ and the words appearing in the outer commutator extensions $\tM_i^k$ depend only on $k$,
and not on the group $G$ or on the tuple $\tN$.
\end{thm}

\begin{proof}
We prove the theorem by induction on $k$.
If $k=1$ then, as already mentioned, it suffices to take the series
\[
V_0^1=\big[[N_1,N_2],[N_1,N_2]\big]\le V_1^1=\big[N_1,[N_1,N_2]\big]\le V_2^1=[N_1,N_2],
\]
so that $v_1^1=v_2^1=\delta_1$, and $M_1^1=(N_1,[N_1,N_2])$ and $M_2^1=(N_1,N_2)$.

Assume now that $k\ge 2$ and that the result is true for $k-1$.
Set $\tN_1=(N_1,\ldots,N_{2^{k-1}})$ and
$\tN_2=(N_{2^{k-1}+1},\ldots,N_{2^{k}})$.
By the induction hypothesis, if we put $s=2^{k-1}+2^{k-2}-1$, then there exist series
\begin{equation}
\label{series(k-1)-1}
V_{0,1}^{k-1}=[\delta_{k-1}(\tN_1),\delta_{k-1}(\tN_1)]
\le
\cdots
\le
V_{i,1}^{k-1}
\le
\cdots
\le
V_{s,1}^{k-1} =\delta_{k-1}(\tN_1)
\end{equation}
and
\begin{equation}
\label{series(k-1)-2}
V_{0,2}^{k-1}=[\delta_{k-1}(\tN_2),\delta_{k-1}(\tN_2)]
\le
\cdots
\le
V_{i,2}^{k-1}
\le
\cdots
\le
V_{s,2}^{k-1} = \delta_{k-1}(\tN_2)
\end{equation}
such that, for every $i=1,\ldots,s$, the factors $V_{i,1}^{k-1}/V_{i-1,1}^{k-1}$ and $V_{i,2}^{k-1}/V_{i-1,2}^{k-1}$ are the images of 
$v_i^{k-1}(\tM_{i,1}^{k-1})$ and $v_i^{k-1}(\tM_{i,2}^{k-1})$, respectively, where:
\begin{enumerate}
\item[(a)]
$v_i^{k-1}$ is an extended word of $\delta_{k-1}$ of degree at most $k-2$.
\item[(b)]
$\tM_{i,1}^{k-1}$ is an outer commutator extension of $\tN_1$.
\item[(c)]
$\tM_{i,2}^{k-1}$ is an outer commutator extension of $\tN_2$.
\item[(d)]
In the sections $V_{i,1}^{k-1}/V_{i-1,1}^{k-1}$ and $V_{i,2}^{k-1}/V_{i-1,2}^{k-1}$, the word $v_i^{k-1}$ is linear in one component
of the tuples $\tM_{i,1}^{k-1}$ and $\tM_{i,2}^{k-1}$, respectively.
\end{enumerate}

Let us now see how to obtain the series for $k$ and for the tuple $\tN$ from the two series \eqref{series(k-1)-1} and \eqref{series(k-1)-2}.
We start by taking the commutator of all terms of the series \eqref{series(k-1)-1} with $\delta_{k-1}(\tN_2)$.
This way we obtain the series
\begin{equation}
\label{series(k-1)-1-comm}
[V_{0,1}^{k-1},\delta_{k-1}(\tN_2)]
\le
\cdots
\le
[V_{i,1}^{k-1},\delta_{k-1}(\tN_2)]
\le
\cdots
\le
[\delta_{k-1}(\tN_1),\delta_{k-1}(\tN_2)]
=
\delta_k(\tN).
\end{equation}
By P.\ Hall's Three Subgroup Lemma, we have
\begin{equation*}
\begin{split}
[V_{0,1}^{k-1},\delta_{k-1}(\tN_2)]
&=
[\delta_{k-1}(\tN_1),\delta_{k-1}(\tN_1),\delta_{k-1}(\tN_2)]
\\
&\le
[\delta_{k-1}(\tN_1),\delta_{k-1}(\tN_2),\delta_{k-1}(\tN_1)]
= 
[\delta_{k-1}(\tN_1),\delta_k(\tN)].
\end{split}
\end{equation*}
Now we multiply all terms of the series \eqref{series(k-1)-1-comm} by
$[\delta_{k-1}(\tN_1),\delta_k(\tN)]$, and this is the rightmost part of the series we are seeking (where $t=2^k+2^{k-1}-1$, as above):
\begin{multline}
\label{seriesk-1}
V_{t-s}^k
=
[\delta_{k-1}(\tN_1),\delta_k(\tN)]
\le
\cdots
\le
V_{t-s+i}^k
=
[V_{i,1}^{k-1},\delta_{k-1}(\tN_2)] \, [\delta_{k-1}(\tN_1),\delta_k(\tN)]
\\
\le
\cdots
\le
V_t^k = \delta_k(\tN).
\end{multline}
Note that $t-s=2^{k-1}+2^{k-2}$.
The factors in this series are the images of the subgroups
\[
[v_i^{k-1}(\tM_{i,1}^{k-1}),\delta_{k-1}(\tN_2)],
\]
which can be represented in the form
$v_i^k(\tM_i^k)$ by taking
\[
v_i^k = [v_i^{k-1},\delta_{k-1}(x_{2^{k-1}+1},\ldots,x_{2^k})]
\]
and defining $\tM_i^k$ to be the concatenation of $\tM_{i,1}^{k-1}$ and $\tN_2$, where the elements of $\tN_2$ occupy the positions $2^{k-1}+1,\ldots,2^k$
(which are the positions corresponding to the variables $x_{2^{k-1}+1},\ldots,x_{2^k}$).
Note that $\tM_i^k$ is an outer commutator extension of $\tN$.

In a symmetric way, by first taking the commutator of $\delta_{k-1}(\tN_1)$ with all terms of the series \eqref{series(k-1)-2} and then multiplying by
$[\delta_k(\tN),\delta_{k-1}(\tN_2)]$, we get the series
\begin{multline}
\label{series-k-2 aux}
W_{t-s}^k
=
[\delta_k(\tN),\delta_{k-1}(\tN_2)]
\le
\cdots
\le
W_{t-s+i}^k
=
[\delta_{k-1}(\tN_1),V_{i,2}^{k-1}] \, [\delta_k(\tN),\delta_{k-1}(\tN_2)]
\\
\le
\cdots
\le
W_t^k = \delta_k(\tN).
\end{multline}
In this series, the factors are given by the images of the subgroups $w_i^k(\mathbf{L}_i^k)$, where
\[
w_i^k = [\delta_{k-1}(y_{2^{k-1}+1},\ldots,y_{2^k}),v_{i,2}^{k-1}],
\]
$v_{i,2}^{k-1}$ being the same word as $v_i^{k-1}$, with $x_1,\ldots,x_{2^{k-1}}$ replaced with $x_{2^{k-1}+1},\ldots,x_{2^k}$,
and $\mathbf{L}_i^k$ being the concatenation of $\tM_{i,2}^{k-1}$ and $\tN_1$, where we put the components of the second tuple after the components of the first.
Note that $w_i^k$ is an extended word of $\delta_{k-1}(x_{2^{k-1}+1},\ldots,x_{2^k})$ of degree at most $k-1$.

Now we take the commutator of $\delta_{k-1}(\tN_1)$ with the terms of the last series, and subtract $s$ to all indices, getting
\begin{multline}
\label{seriesk-2 aux-2}
Z_1^k
=
\big[\delta_{k-1}(\tN_1),[\delta_k(\tN),\delta_{k-1}(\tN_2)]\big]
\le
\cdots
\le
Z_{i+1}^k
=
[\delta_{k-1}(\tN_1),W_{t-s+i}^k]
\\
\le
\cdots
\le
Z_{t-s}^k = [\delta_{k-1}(\tN_1),\delta_k(\tN)],
\end{multline}
since $t-2s=1$.
Finally, we define $V_0^k=[\delta_k(\tN),\delta_k(\tN)]$ and multiply all terms of \eqref{seriesk-2 aux-2} with this subgroup, setting
$V_i^k=Z_i^kV_0^k$ for $i=1,\ldots,t-s$.
Since $V_0^k\le [\delta_{k-1}(\tN_1),\delta_k(\tN)]$, we get the series
\begin{multline}
\label{seriesk-2}
V_0^k = [\delta_k(\tN),\delta_k(\tN)]
\le
\cdots
\le
V_{i+1}^k
=
Z_{i+1}^k [\delta_k(\tN),\delta_k(\tN)]
\\
\le
\cdots
\le
V_{t-s}^k = [\delta_{k-1}(\tN_1),\delta_k(\tN)].
\end{multline}
In this series, the factors $V_i^k/V_{i-1}^k$ for $i=2,\ldots,t-s$ are given by the images of the subgroups $v_i^k(\mathbf{M}_i^k)$, where
\[
v_i^k = [\delta_{k-1}(x_1,\ldots,x_{2^{k-1}}),w_{i+s}^k]
\]
is an extended word of $\delta_k$ of degree at most $k-1$, and $\tM_i^k$ is the concatenation of $\tN_1$ and $\mathbf{L}_{i+s}^k$, with
the components of the second tuple after the components of the first.
On the other hand, the quotient $V_1^k/V_0^k$ is given by the image of $v_1^k(\tM_1^k)$, where
\[
v_1^k = [\delta_{k-1}(x_1,\ldots,x_{2^{k-1}}),[y,\delta_{k-1}(x_{2^{k-1}+1},\ldots,x_{2^k})]]
\]
and $\tM_1^k=(\tN,\delta_k(\tN))$.

Now the concatenation of \eqref{seriesk-1} and \eqref{seriesk-2} is the desired series for $k$ and $\tN$.
The discussion of the previous paragraphs shows that $v_i^k(\tM_i^k)$ is an extension of $\delta_k(\tN)$ for every $i=1,\ldots,t$.
Thus we only need to check linearity of every word $v_i^k$ in one component of the vector $\tM_i^k$.
For $i=t-s+1,\ldots,t$, if $v_i^{k-1}$ is linear in component $j$ of $\tM_{i,1}^{k-1}$ of the initial series \eqref{series(k-1)-1}, then combining this fact with \cref{linearity after commutator}, it follows that $v_i^k$ is linear in the same component of $\tM_i^k$.
For $i=1,\ldots,t-s$, we can use similarly the linearity of the series \eqref{series(k-1)-2}. 
Finally for $i=1$, since $V_0^k=[\delta_k(\tN),\delta_k(\tN)]$, we have linearity in the component corresponding to $y$, which takes values in $\delta_k(\tN)$.
\end{proof}

We can now prove the corresponding version of \cref{gamma_r concise most general} for the words $\delta_k$.

\begin{thm}
\label{delta_k concise most general}
Let $k\in\N$ and let $\tN=(N_1,\ldots,N_{2^k})$ be a tuple of normal subgroups of a group $G$.
Assume that $N_i=\langle S_i \rangle$ for every $i=1,\ldots,2^k$, where:
\begin{enumerate}
\item
$S_i$ is a normal subset of $G$.
\item
There exists $n_i\in\N$ such that all $n_i$th powers of elements of $N_i$ are contained in $S_i$.
\end{enumerate}
If for the tuple $\tS=(S_1,\ldots,S_{2^k})$ the set of values $\delta_k\{\tS\}$ is finite of order $m$, then the subgroup $\delta_k(\tN)$ is also finite and
of $(m,k,n_1,\ldots,n_{2^k})$-bounded order.
\end{thm}

\begin{proof}
Let us consider the series
\[
[\delta_k(\tN),\delta_k(\tN)] = V_0^k\le V_1^k\le \cdots
\le V_t^k = \delta_k(\tN)
\]
of \cref{linear series delta}, where $t=2^k+2^{k-1}-1$.
We prove that every $V_i^k$ is finite of bounded order by induction on $i$.
The result for $i=0$ follows from \cref{w(N)' finite}.
Assume now that $i\ge 1$ and that the result is true for $V_{i-1}^k$.
By \cref{linear series delta}, the section $V_i^k/V_{i-1}^k$ coincides with the image of a subgroup $v_i^k(\tM_i^k)$ that is an extension of $\delta_k(\tN)$ of degree at most $k-1$.

For simplicity, let us write $v$ for $v_i^k$ and $\tM$ for $\tM_i^k$.
Let $\tM=(M_1,\ldots,M_s)$, which is an outer commutator extension of $\tN$.
Hence $s\ge 2^k$ and for every $j=1,\ldots,s$ we have $M_j=w_j(\tN_j)$, where $w_j$ is an outer commutator word, all components in $\tN_j$ belong to $\tN$, and one of these components must be $N_j$ for $j=1,\ldots,2^k$.

Let $T_j=w_j\{\tS_j\}$, where $\tS_j$ is obtained from $\tN_j$ by replacing each subgroup $N_{\ell}$ with its given generating set $S_{\ell}$.
Hence $T_j\subseteq M_j$.
Recall from \cref{linear series delta} that the word $v$ (and hence also the number $s$ of variables of $v$) and the words $w_1,\ldots,w_s$ only depend on $k$,
and not on $G$ or on $\tN$.
From this fact, and since $\tS_j$ consists of normal subsets of $G$, it follows from \cref{one value in ocw} that there is a function
$h:\N\rightarrow\N$ such that $T_j\subseteq S_j^{\ast h(k)}$ for every $j=1,\ldots,2^k$.
If we set $\tT=(T_1,\ldots,T_s)$ then, by \cref{width extended}, we get $v\{\tT\}\subseteq w\{\tS\}^{\ast n}$, where
\[
n = h(k)^{2^k} 2^{k-1}.
\]
Consequently
\begin{equation}
\label{bound v(T)}
|v\{\tT\}| \le (2m+1)^n,
\end{equation}
and $v\{\tT\}$ is finite of $(m,k)$-bounded cardinality.
On the other hand, it follows from \cref{generators ocw normal subgroups} that $v(\tM)$ can be generated by the set of values $v\{\tT\}$.

From \cref{linear series delta}, we know that the word $v$ is linear in some position $j\in\{1,\ldots,s\}$ of the tuple $\tM$ modulo $V_{i-1}^k$.
Since $M_j=w_j(\tN_j)$ is as above, we have $M_j\le N_{\ell}$ for some $\ell\in\{1,\ldots,2^k\}$, and actually $\ell=j$
if $j\in\{1,\ldots,2^k\}$.
Now, from linearity, for every tuple $\mathbf{t}=(t_1,\ldots,t_s)\in \tT$ and every $\lambda\in\Z$, we have
\begin{equation}
\label{power out of v}
v(\mathbf{t})^{\lambda n_{\ell}}
=
v(t_1,\ldots,t_j,\ldots,t_s)^{\lambda n_{\ell}}
\equiv
v(t_1,\ldots,t_j^{\lambda n_{\ell}},\ldots,t_s)
\pmod{V_{i-1}^k}.
\end{equation}
We have $t_j^{\lambda}\in M_j\le N_{\ell}$ and then, by (ii) in the statement of the theorem, $t_j^{\lambda n_{\ell}}\in S_{\ell}$.
So we get
\[
v(t_1,\ldots,t_j^{\lambda n_{\ell}},\ldots,t_s) \in v\{\tT_j\},
\]
where $\tT_j$ is the tuple obtained from $\tT$ after replacing $T_j$ with $S_{\ell}$.
Similarly to \eqref{bound v(T)} and taking into account that $\ell=j$ if $j\in\{1,\ldots,2^k\}$, it follows that the set $v\{\tT_j\}$ is finite of $(m,k)$-bounded cardinality.
Hence there exist $(m,k)$-bounded integers $\lambda$ and $\mu$, with $\lambda\ne\mu$, such that
\[
v(\mathbf{t})^{\lambda n_{\ell}}
\equiv
v(\mathbf{t})^{\mu n_{\ell}}
\pmod{V_{i-1}^k}.
\]
This implies that $v(\mathbf{t})$ has finite order modulo $V_{i-1}^k$, bounded in terms of $m$, $k$ and $n_{\ell}$.

Summarizing, the abelian quotient $V_i^k/V_i^{k-1}$ is the image of the verbal subgroup $v(\tM)$, which is generated by the set $v\{\tT\}$
of $(m,k)$-bounded cardinality, and each element of $v\{\tT\}$ has $(m,k,n_{\ell})$-bounded order.
We conclude that the order of $V_i^k/V_{i-1}^k$ is $(m,k,n_{\ell})$-bounded, which completes the proof.
\end{proof}

Exactly as in the case of lower central words, we obtain Theorems A and B as special cases of this last result.

\begin{cor}
\label{delta_k of non-comm concise}
Let $k\in\N$ and let $u_1,\ldots,u_{2^k}$ be non-commutator words.
Then the word $\delta_k(u_1,\ldots,u_{2^k})$ is boundedly concise.
In particular, $\delta_k(x_1^{n_1},\ldots,x_{2^k}^{n_{2^k}})$ is boundedly concise for all $n_i\in \Z\smallsetminus\{0\}$.
\end{cor}

\begin{cor}
\label{delta_k concise on normal}
Let $k\in\N$ and let $\tN=(N_1,\ldots,N_{2^k})$ be a tuple of normal subgroups of a group $G$.
If $\delta_k\{\tN\}$ is finite of order $m$, then the subgroup $\delta_k(\tN)$ is also finite, of $(m,k)$-bounded order.
\end{cor}


\vskip 0.4 true cm

\begin{center}{\textbf{Acknowledgments}}
\end{center}
Supported by the Spanish Government, grant PID2020-117281GB-I00, partly with FEDER funds, and by the Basque Government, grant IT483-22.
The second author is also supported by a grant FPI-2018 of the Spanish Government.

The authors thank the anonymous referee for helpful suggestions that allowed them to produce an improved version of the paper. 

\vskip 0.4 true cm



\bigskip
\bigskip

{\footnotesize \pn{\bf Gustavo A.\ Fern\'andez-Alcober}\;
\\
{Department of Mathematics}, {University of the Basque Country UPV/EHU,} {Bilbao, Spain}\\
{\tt Email: gustavo.fernandez@ehu.eus}\\

{\footnotesize \pn{\bf Matteo Pintonello}\;
\\
{Department of Mathematics}, {University of the Basque Country UPV/EHU,} {Bilbao, Spain}\\
{\tt Email: matteo.pintonello@ehu.eus}\\

\end{document}